\documentclass[11pt,a4paper]{amsart}		

\usepackage[applemac]{inputenc}					% for unicode support

\usepackage{lmodern}						% font
\usepackage[english]{babel}					% babel
\usepackage[weather]{ifsym}					% for visible angle

\usepackage{tikz}
\usetikzlibrary{calc, svg.path, patterns}
\usepackage{todonotes}

\usepackage{amsmath,amsfonts,amssymb,amsthm}
\usepackage[left=3cm,right=3cm,top=3cm,bottom=3cm]{geometry}

\usepackage{mathtools}						% uncomment to tag only the referred equations (post-production)
\newtheorem{theorem}{Theorem}[section]
\newtheorem{lemma}{Lemma}[section]
\newtheorem{proposition}{Proposition}[section]

\theoremstyle{remark}

\newcommand\mydef{\mathrel{\stackrel{\makebox[0pt]{\mbox{\small def}}}{=}}}

\title[Recovering nonlinear magnetic potential from partial data on Riemann surface]
{A Note on the Partial data inverse problem for a nonlinear magnetic Schr\"odinger operator on Riemann Surface}

\begin{document}

\author[Y. Ma]{Yilin Ma}
\address[Y. Ma]%
{F07-CARSLAW BUILDING, THE UNIVERSITY OF SYDNEY, CHIPPENDALE NSW, AUSTRALIA}
\email{K.Ma@maths.usyd.edu.au}

\begin{abstract}
We recover a nonlinear magnetic Schr\"odinger potential from measurement on an arbitrarily small open subset of the boundary on a compact Riemann surface. We assume that the magnetic potential satisfies suitable analytic properties, in which case the recovery can be obtained by a linearisation argument. The proof relies on the complex analytic methods introduced in \cite{leocalderon}.
\end{abstract}

\maketitle


\setcounter{equation}{0}

\medskip

\noindent {\bf Acknowledgments.} This work is completed as part of the author's HDR degree while under the support of A/Prof Tzou's projects ARC DP190103451 and ARC DP190103302. The author is grateful for the supervision he has received. 
\medskip

\section{Introduction}

Let $(M,g)$ be a compact Riemann surface with smooth boundary $\partial M$. For a complex parameter $\zeta \in \mathbb{C}$ and some $\alpha > 0$, we let $(X(\zeta, \cdot) ,V(\zeta, \cdot) )$ be in $\mathcal{C}^{2,\alpha}(T^{\star}M) \times \mathcal{C}^{1,\alpha}(M)$ depending analytically in $\zeta$, in the sense that
\begin{alignat}{2} \label{holomorphic condition} 
X(\zeta,p) \ \mydef \ \sum_{k \geq 0} \frac{\zeta^{k}}{k!} X_k(p) \ \ \text{and} \ \ V(\zeta,p) \ \mydef \ \sum_{k \geq 0} \frac{\zeta^{k}}{k!} V_{k}(p)
\end{alignat}
with convergence respectively in the $\mathcal{C}^{2,\alpha}(T^{\star}M)$ and $\mathcal{C}^{1,\alpha}(M)$ topologies and $p \in M$. For each $k \geq 0$, the pair $(X_{k}, V_{k})$ belongs to $ \mathcal{C}^{2,\alpha}(T^{\star}M) \times \mathcal{C}^{1,\alpha}(M) $, and we may consider a differential operator
\begin{alignat}{2} \label{magnetic schordinger operator}
L_{X,V}u \ \mydef \ - \star ( d \star + i X(u,p) \wedge \star )(d+ i X(u,p))u + + V(u,p)
\end{alignat}
where $\star$ denotes the Hodge star operator on $(M,g)$. One refers to $L_{X,V}$ as the \emph{magnetic Schr\"odinger operator} with magnetic potential $(X,V)$. It is known from \cite{guntherNMS} that if $X_0 = 0$ and $V_{0} = V_{1} = 0$, then there exists $\delta > 0$ small and a constant $C>0$ such that if we set 
\begin{alignat*}{2}
& U_{\delta} \ \mydef \ \{ f \in \mathcal{C}^{2,\alpha}(\partial M) \ / \ \| f \|_{\mathcal{C}^{2,\alpha}(\partial M)} <\delta \}\ \ \text{and} \\
 & \hspace*{3.8mm} V_{C,\delta}  \ \mydef \ \{ u \in \mathcal{C}^{2,\alpha}(M) \ / \ \| u \|_{\mathcal{C}^{2,\alpha}(M)} < C \delta \},
\end{alignat*} 
then for every $f \in U_{\delta}$ there exists $u_f \in V_{C, \delta}$ with ${u_{f}}_{|_{\partial M}} = f$ such that $L_{X,V}u_f = 0$. As such we may define the \emph{Dirichlet-Neumann map} by
\begin{alignat*}{2}
\Lambda_{X,V}:
\begin{cases}
U_{\delta} \rightarrow \gamma\nabla^{X}_{\nu}(V_{C,\delta}), \\
\hspace*{1.7mm} f \mapsto {\nabla^{X}_{\nu} u_f }_{|_{\partial M}},
\end{cases}
\end{alignat*}
where $ \nabla^{X}u = du + i uX(u,p) $, $ \nabla^{X}_{\nu}u = \nabla^{X}u(\nu) $ by definition where $\nu$ is the outward point unit normal, and $\gamma$ is the boundary trace map. The inverse problem asks whether $\Lambda_{X,V}$ uniquely determines $(X,V)$. However, in practice one might only have access to measurements on a small portion of the boundary. Hence if $\Gamma \subset \partial M$ is an arbitrary non-empty open subset, one could asks whether the \emph{partial data Dirichlet-Neumann map}
\begin{alignat*}{2}
\Lambda^{\Gamma}_{X,V}:  
\begin{cases}
U_{\delta} \cap \mathcal{E}'(\Gamma) \rightarrow \gamma \nabla^{X}_{\nu} (V_{C,\delta}) \cap \mathcal{E}'(\Gamma) \\
\hspace*{15mm} f \mapsto {\nabla^{X}_{\nu} u_f}_{|_{\Gamma}}
\end{cases}
\end{alignat*}
determines $(X,V)$. In this article we will give positive answer to this inverse problem, the precise formulation of which is the following:
\begin{theorem} \label{main theorem}
Let $ (X_j(\zeta, \cdot ), V_j(\zeta, \cdot)) \in \mathcal{C}^{2,\alpha}(T^{\star}M) \times \mathcal{C}^{1,\alpha}(M)$, $j=1,2$ be pairs depending holomorphically on a complex parameter $\zeta \in \mathbb{C}$ in the sense that
\begin{alignat}{2} \label{holomorphic condition 2} 
X_j(\zeta,p) \ \mydef \ \sum_{k \geq 0} \frac{\zeta^{k}}{k!} X_{j,k}(p) \ \ \text{and} \ \ V_j(\zeta,p) \ \mydef \ \sum_{k \geq 0} \frac{\zeta^{k}}{k!} V_{j,k}(p).
\end{alignat}
Assume that $X_{j,0} = 0$ and $V_{j,0} = V_{j,1} = 0$. If $ \Lambda^{\Gamma}_{X_1, V_1} = \Lambda^{\Gamma}_{X_2,V_2} $ where $\Gamma$ is an arbitrarily small non-empty open subset of $\partial M$, then $(X_1, V_1) = (X_2, V_2)$.
\end{theorem}
The Calder\'on problem for elliptic equations has been intensively studied since the groundbreaking result \cite{calderonfirstresult} of Sylvester-Uhlmann. Since them, much attentions have been paid to the study of Calder\'on problem for the linear magnetic Schr\"odinger operator. See \cite{21,19,5,14,22} and in particular \cite{leoconnection,reflectionpaper} for results on Riemann surface. The list is not exhaustive. In the case where $X = 0$ and a nonlinear potential $V(u,p)$ is present. Lassas-Liimatainen-Lin-Salo observed in \cite{Lassasnonlinear} that recovery is possible provided $V(u.p) = u^{k}V(p)$, $k \geq 2$ is of power type and the underlying geometry is Euclidean with dimension $n \geq 2$. By linearising the Dirichlet-Neumann map, they showed that the recovery of $V$ amounts to solving the linearised Calder\'on problem. By taking successive linearisations, Krupchyk-Uhlmann in \cite{guntherNLC} gave positive answer to the partial data inverse problem of a nonlinear $V(u,p)$ satisfying condition (\ref{holomorphic condition}) with $V_{0} = V_{1} = 0$. The same was done on Stein manifolds with K\"ahler metric by Ma-Tzou in \cite{Ma}. If in addition $V_2$ is known, then Feizmohammadi-Oksanen showed in \cite{CTANCP} that the full data inverse problem for $V(u,p)$ can be solved on conformally anisotropic manifolds (CTA) with admissible transversed manifolds. \par 
In the case where $X \neq 0$ and is nonlinear, the boundary integral identities obtained after linearising the Dirichlet-Neumann maps involve first order derivatives which is different from the standard linearised Calder\'on problem. Under the condition of vanishing lower order terms, Lai-Zhou in their recent article \cite{LaiNMS} recovered $(X,V)$ from partial data in the Euclidean settings for all dimensions $n \geq 2$. In another result \cite{guntherNMS}, Krupchyk-Uhlmann solved the relevant full data problem on CTA manifolds of dimensions $n \geq 3$. Since the Calder\'on problem for the linear Schr\"odinger equation in two dimensions has historically been reliant on complex analytic techniques (see \cite{leocalderon} and the references therein), it is encouraging to see whether these powerful methods can be applied to the current problem and in particular fill the gap for the case of two dimensional non-Euclidean geometry. As it turns out, in this case the key obstacle encountered by \cite{guntherNMS, LaiNMS} is resolved by an incredibly short argument.
\section{Recovery from Integral Identities}
\noindent In this section we state the main recovery results to be used in the proof of Theorem \ref{main theorem}. 
\begin{proposition} \label{identification}
Let $X,Q \in \mathcal{C}^{1,\alpha}(T^{\star}M)$ be such that
\begin{alignat}{2} \label{boundary integral identity 1}
\int_{M} u_1 u_2 X \wedge \star du_{3} = \int_{M} u_1 u_2 u_3 Q
\end{alignat}
for all harmonic functions $u_1, u_2$ and $u_3$ which vanish on $\partial M \backslash \Gamma$, then $X = Q = 0$. 
\end{proposition}
\noindent Our starting point is the following result on the linearised Calder\'on problem. 
\begin{proposition} \label{linear Calderon problem}
Let $f \in \mathcal{C}^{1,\alpha}(M)$ be such that 
\begin{alignat*}{2}
\int_{M} u_1 u_2 f \hspace*{0.5mm} d\text{v}_{g}  = 0
\end{alignat*}
for all harmonic functions $u_1, u_2$ which vanish on $\partial M \backslash \Gamma$, then $f=0$.
\end{proposition}
\noindent We refer to \cite[Proposition 5.1]{leocalderon} for a proof of this claim. Applying the above to (\ref{boundary integral identity 1}), we deduce immediately 
\begin{alignat}{2} \label{wedge product identification}
 X \wedge \star du = u Q
\end{alignat} 
for all harmonic function $u$ which vanish on $\partial M \backslash \Gamma$. We also remark that
\begin{lemma}[Riemann-Lebesgue] \label{Riemann-Lebesgue}
Let $\psi \in \mathcal{C}^{\infty}(M)$ be a real valued function for some $N$ large. For every $f \in L^1(M)$ we have as $\lambda \rightarrow \infty$ 
\begin{alignat*}{2}
\int_{M} e^{i\lambda \psi} f \hspace*{0.5mm} d\text{v}_{g} = o(1).
\end{alignat*}
\end{lemma}
\noindent The proof is standard and can be found in \cite[Lemma 5.3]{leocalderon}. Now we prove
\begin{lemma} \label{identification result 1}
Let $F$ be a holomorphic function which is purely real on $\partial M \backslash \Gamma$. Let $A, B$ be respectively the holomorphic and anti-holomorphic parts of $X$, then we have
\begin{alignat*}{2} 
A \wedge \partial F = 0 \ \ \text{and} \ \ B \wedge \bar{\partial} \bar{F} = 0.
\end{alignat*}
\end{lemma}
\begin{proof}
We only prove the case of $A$ since the other claim follows from conjugation. Let $v$ be any $L^{1}(M)$ function and $\lambda > 0$. Since $e^{\lambda F}$ is purely real on $\partial M \backslash \Gamma$ and holomorphic, the function $e^{\lambda F} - e^{\lambda \bar{F}}$ is harmonic and vanish on $\partial M \backslash \Gamma$. Using (\ref{wedge product identification}) we deduce that 
\begin{alignat}{2} 
\int_{M} v e^{-\lambda F}( e^{\lambda F} - e^{\lambda \bar{F}} ) & \hspace*{0.5mm} Q = \int_{M} v e^{-\lambda F} X \wedge \star d (e^{\lambda F} - e^{\lambda \bar{F}}) \label{divide this equation} \\ 
& = \int_{M} v e^{-\lambda F} A \wedge i \partial e^{\lambda F} + v e^{-\lambda F} B \wedge i \bar{\partial} e^{\lambda \bar{F}} \nonumber \\
& = \int_{M} \lambda v A \wedge i \partial F + \lambda  v e^{- \lambda (F - \bar{F})} B \wedge i \bar{\partial} \bar{F}. \nonumber
\end{alignat}
On the other hand
\begin{alignat*}{2}
\int_{M} v e^{-\lambda F}(e^{\lambda F} - e^{\lambda \bar{F}}) \hspace*{0.5mm} Q = \int_{M} v Q - v e^{-\lambda(F - \bar{F} )} \hspace*{0.5mm} Q,
\end{alignat*}
thus by Lemma \ref{Riemann-Lebesgue}, we can divide (\ref{divide this equation}) by $i\lambda$ everywhere to get that
\begin{alignat*}{2}
\int_{M} v A \wedge \partial F = o(1) + \frac{O(1)}{\lambda}
\end{alignat*}
as $\lambda \rightarrow \infty$. Taking this limit gives us
\begin{alignat*}{2}
\int_{M} v A \wedge \partial F = 0.
\end{alignat*}
Since $v$ is an arbitrary function in $L^{1}(M)$ we must have $A \wedge \partial F = 0$.
\end{proof}
\begin{proof}[Proof of Proposition \ref{identification}]The existence of a holomorphic function $F$ with the required boundary condition is ensured by \cite[Proposition 2.1]{leocalderon}. Since locally in a holomorphic coordinate we can write $A = \alpha d\bar{z}$ and $| A \wedge \partial F| = |\alpha| |\partial_{z}F|$ where $\partial_{z}F$ is harmonic, we can use the result of \cite{measurezero} which state that the zeros of harmonic functions has measure zero to conclude from Lemma \ref{identification result 1} that $A = 0$. Similarly we can deduce $B = 0$. Now we have $ uQ = 0 $ for any harmonic function $u$ which vanish on $\partial M \backslash \Gamma$, but the same reason we must have $Q = 0$.
\end{proof}
We will also need the following result on boundary determination which follows from a rather standard calculation. 
%\begin{proposition} \label{integral identity for boundary determination}
%Let $X \in \mathcal{C}^{2,\alpha}(T^{\star}M)$, $Q \in L^{\infty}(T^{\star} M)$ be such that 
%\begin{alignat*}{2} 
%\int_{M} X \wedge \star u_1 du_2 - X \wedge \star u_2 du_1 = \int_{M} u_1 u_2 Q 
%\end{alignat*}
%for all harmonic functions $u_1$ and $u_2$ which vanish on $\partial M \backslash \Gamma$. Let $\tau$ be the tangential vector field along $\partial M$, then $X(\tau) = 0$ along $\Gamma$. 
%\end{proposition}
%\noindent This statement is shown in [] where the underlying geometry is Euclidean. However, since their method is purely local the proof carries naturally to the case of Riemann surface. A sketch of this more general argument is provided in [] when full data measurement is involved. A related result is 
\begin{proposition} \label{integral identity for boundary determination 2}
Let $X \in \mathcal{C}^{2,\alpha}(T^{\star}M), Q \in L^{\infty}(T^{\star}M)$ be such that 
\begin{alignat}{2} \label{integral identity to be used for boundary determination}
\int_{M} u_1 X \wedge \star du_2 = \int_{M} u_1 u_2 Q
\end{alignat}
for all harmonic functions $u_1$ and $u_2$ which vanish on $\partial M \backslash \Gamma$. Let $\tau$ be the unit tangential vector field along $\partial M$, then $X(\nu) = X(\tau) = 0$ along $\Gamma$.
\end{proposition}
\noindent Let $\alpha = ( i, - 1 )$ be chosen such that $\alpha.\alpha = 0$. Near an arbitrary $p_0 \in \Gamma$ we find conformal boundary normal coordinate $z = x + iy$ centered at $p_0$ where $|z| \leq 1$, $y > 0$ and the boundary is given by $\{ x = 0 \}$. In such a chart the metric is given by $g = e^{2\rho} |dz|^2$ for some smooth function $\rho$ and $(\partial_{x}^2 + \partial_{y}^2) e^{\alpha \cdot z /h} = 0$. Hence the conformal covariance of the Laplacian ensures that $\Delta_{g} e^{\alpha \cdot z /h} = 0$. Let $\eta$ be a smooth cut-off function that is identically $1$ for $|z| < 1/2$ and $0$ for $|z| > 3/4$. Set $\eta_{h}(z) = \eta(z/\sqrt{h})$ and $v_{h}(z) = \eta_{h} e^{\alpha \cdot z/h}$. 
\begin{lemma} \label{estimates for harmonic function}
Let $v_{h}$ be defined as above, then there exists $(r_1,r_2) \in L^2(M) \times H^1(M)$ such that $\| r_1 \|_{L^2(M)} \leq C h^{5/4}$ and $\| r_2 \|_{H^1(M)} \leq C h^{1/4}$. Moreover, the functions $u_1 = v_{h} +r_1$ and $u_2 = v_{h} +r_2$ are harmonic. In particular, $u_1$ and $u_2$ vanish on $\partial M \backslash \Gamma$.
%\begin{alignat*}{2}
%&\| v_{h} \|_{L^{2}(M)} \leq Ch^{3/4}, \ \ \| d v_{h} \|_{L^{2}(T^{\star}M)} \leq C h^{-1/4}, \\
%& \| \Delta_{g} v_h \|_{H^{-1}(M)} \leq C h^{1/4}, \ \ \| \Delta_{g}v_{h} \|_{H^{-2}(M)} \leq C
%\end{alignat*}
\end{lemma}
\begin{proof}
The construction of $r_1$ is covered in \cite[Lemma 7.2]{leocalderon}. We extend their argument to construct $r_2$. First we have the computation
\begin{alignat}{2}  \label{laplacian calculation}
\Delta_{g} v_{h}(z) = \frac{1}{h} e^{\alpha \cdot z/h}( \Delta_{g}\eta )(z/ \sqrt{h}) + \frac{2}{h^{3/2}} e^{\alpha \cdot z /h} \langle d\eta( z / \sqrt{h} ), \alpha.dz  \rangle_g.
\end{alignat}
Let $w \in H^{1}_0(M)$. For $\chi = -i \partial_{x} \eta +\partial_{y}\eta$ and $\chi_{h}(z) = \chi(z/\sqrt{h})$ we have 
\begin{alignat}{2}
\frac{2}{h^{3/2}} \int_{|z|< \sqrt{h}} e^{\alpha \cdot z/h} &\langle d \eta(z/\sqrt{h}), \alpha.dz \rangle_g \hspace*{0.5mm} w d\text{v}_g = \frac{1}{h^{1/2}} \int_{|z|< \sqrt{h}}  \chi_h w (i \partial_{x} + \partial_{y}) e^{\alpha \cdot z /h} \hspace*{0.5mm} dxdy \nonumber \\
& = - \frac{1}{h^{1/2}} \int_{|z|< \sqrt{h} } e^{\alpha \cdot z/h} w (i\partial_{x} + \partial_y)(\chi_h) e^{2\rho} \hspace*{0.5mm}dxdy  \label{first integrla} \\
& \hspace*{1.5cm} - \frac{1}{h^{1/2}} \int_{|z|< \sqrt{h} } e^{\alpha \cdot z/h} \chi_h (i \partial_x +\partial_y) (w e^{2\rho}) \hspace*{0.5mm} dxdy, \label{second integral}
\end{alignat}
there is no boundary term in the integration by parts since $w_{|_{\partial M}} = 0$. Integral (\ref{second integral}) can be easily estimate by $ C h^{1/4} \| w \|_{H^1}  $. However, the differential operator in (\ref{first integrla}) will cause the appearance of an extra order $h^{-1/2}$ in front of the integral, so we have only improved the growth from $h^{-3/2}$ to $h^{-1}$. But as long as $w$ is not differentiated, we can apply the same integration by parts argument to obtain improvement as above. By doing so we can get that (\ref{first integrla}) is bounded by $Ch^{1/4} \|w\|_{H^1}$ as well. The other term in (\ref{laplacian calculation}) can be estimated in a similar fashion. This shows that $\| \Delta_{g} v_{h} \|_{H^{-1}} \leq C h^{1/4}$. Now by the standard Lax-Milgram argument we can find $r_2 \in H^{1}_0$ such that $\Delta_{g}r_2 = - \Delta_{g}v_h$ and $\| r_2 \|_{H^{1}} \leq C \| \Delta_{g}v_{h} \|_{H^{-1}}$ and so our claim is proved.
\end{proof}
\begin{proof}[Proof of Proposition \ref{integral identity for boundary determination 2}]
By an elementary calculation we have $\| v_h \|_{L^{2}} \leq Ch^{3/4}$ and $\| dv_{h} \|_{L^{2}} \leq C h^{-1/4}$. Let $u_1$ and $u_2$ be constructed as in Lemma \ref{estimates for harmonic function}, then we easily get
\begin{alignat}{2}
\int_{M} u_1 X \wedge \star d\bar{u}_2 & = \int_{M} v_h X \wedge \star d\bar{v}_h + o(h^{1/2})  = \frac{1}{h} \int_{|z|< \sqrt{h} } e^{-2y/h} \eta_{h}^2 X \wedge \star \bar{\alpha}.dz + o(h^{1/2})  \label{sum of integrals}
\end{alignat}
On the other hand, it is clear that
\begin{alignat*}{2}
\int_{M} u_1 X \wedge \star d\bar{u}_2 =  \int_{M} u_1 \bar{u}_2 Q = o(h^{1/2}).
\end{alignat*}
Moreover, we can do the same by replacing $u_1, u_2$ with their conjugates, so we have
\begin{alignat}{2} \label{add this equation}
\frac{1}{h} \int_{|z|< \sqrt{h} } e^{-2y/h} \eta_h^2 X \wedge \star \bar{\alpha}. dz = o(h^{1/2}) \ \ \text{and} \ \ \frac{1}{h} \int_{|z|< \sqrt{h} } e^{-2y/h} \eta_h^2 X \wedge \star \alpha.dz = o(h^{1/2}). 
\end{alignat}
Let $f = e^{2\rho}\langle X, dy \rangle_g $. Adding the two equations in (\ref{add this equation}) together and a standard integration by parts calculation gives us
\begin{alignat*}{2} 
0 = \frac{2}{h} \int_{|x|< \sqrt{h} } & \int_{0}^{\sqrt{h}} e^{-2y/h} \eta_h^2 f \hspace*{0.5mm} dydx = 2 \int_{|x|<1} \int_{0}^{1} e^{-2y/\sqrt{h}} \eta^2 f(h^{1/2}z) \hspace*{0.5mm} dydx \nonumber \\
&= - h^{1/2} \int_{|x|<1} \eta(x) f(h^{1/2}x) \hspace*{0.5mm} dx + o(h^{1/2}) = C h^{1/2} f(p_0) +o(h^{1/2})
\end{alignat*}
for some $C >0$. Dividing over by $h^{1/2}$ and taking the limit as $h \rightarrow 0$ yields $f(p_0) = 0$. Since $p_0 \in \Gamma$ is chosen arbitrarily we must have $X(\nu)$ along $\Gamma$. Taking the difference of (\ref{sum of integrals}) and applying the same argument also shows $X(\tau) = 0$ along $\Gamma$.
\end{proof}
\section{Proof of the theorem}
\noindent In this section we give the proof of Theorem \ref{main theorem}. 
\begin{proof}
We will proceed via induction. For all $k \geq 1$ we set
\begin{alignat*}{2}
X_{k} \ \mydef \ X_{k,1} - X_{k,2} \ \ \text{and} \ \ V_{k} \ \mydef \ V_{k,1} - V_{k,2}.
\end{alignat*}
\subsection{The Set Up}
Let $f_1$ and $f_2$ be elements of $ \mathcal{C}^{2,\alpha}(\partial M) \cap \mathcal{E}'(\Gamma) $, and for $ \epsilon = ( \epsilon_1, \epsilon_2) \in \mathbb{C}^{2} $ we may set $f_{\epsilon} = \epsilon_1 f_1 + \epsilon_2 f_2$. It follows that if $|\epsilon|$ sufficiently small, then $f_{\epsilon}$ belongs to the domain of $\Lambda^{\Gamma}_{X_j, V_j}$, $j=1,2$. Let $ u_{j,f_{\epsilon}} $ be the unique small solution to $L_{X_{j},V_j} u_{j,f_{\epsilon}} = 0$ with boundary condition $f_{\epsilon}$. Differentiating $u_{j,f_{\epsilon}}$ with respect to $\epsilon_{\ell}$ and evaluating at $\epsilon = 0$ yields $\Delta_{g} \partial_{\epsilon_{\ell}} {u_{j,f_{\epsilon}}}_{|_{\epsilon = 0}} = 0$ and $ \partial_{\epsilon_{\ell}} {u_{j,f_{\epsilon}}}_{|_{\epsilon = 0}} = f_{\ell} $ on $\partial M$. By uniqueness of harmonic functions we may denote $u_{\ell} = \partial_{\epsilon_{\ell}} {u_{1,f_{\epsilon}}}_{|_{\epsilon = 0}} = \partial_{\epsilon_{\ell}} {u_{2,f_{\epsilon}}}_{|_{\epsilon=0}}$. Since ${u_{j,f_{\epsilon}}}_{|_{\epsilon = 0}} = 0$ on $\partial M$, any term involving positive powers of $u_{j,f_{\epsilon}}$ in ${\partial_{\epsilon_1} \partial_{\epsilon_2} ( L_{X_j,V_j} - \Delta_{g})u_{j,f_{\epsilon}}}_{|_{\epsilon = 0}}$ must vanish as well. Thus by a direct calculation we now get 
\begin{alignat}{2} \label{harmonic equation 1}
\begin{cases}
\Delta_{g} \partial_{\epsilon_1} \partial_{\epsilon_2} {u_{j,f_{\epsilon}}}_{|_{\epsilon = 0}} = 3i \star X_{j,1} \wedge \star d(u_1 u_2) - u_1 u_2 (2i d^{\star} X_{j,1} + V_{j,2})  \ \text{in} \ M \\
\hspace*{0.49cm} \partial_{\epsilon_1} \partial_{\epsilon_2} {u_{j,f_{\epsilon}}}_{|_{\epsilon = 0}} = 0 \ \text{on} \ \partial M.
\end{cases}
\end{alignat}
Similarly, one has 
\begin{alignat}{2} \label{linearised equation for forms}
\partial_{\epsilon_1} \partial_{\epsilon_2} i u_{j,f_{\epsilon}} {\langle X_j(u_{j,f_{\epsilon}}, p), \nu \rangle_{g}}_{|_{\epsilon=0}} = 2 i u_1 u_2 X_{j,1}(\nu) \ \ \text{on} \ \ \partial M.
\end{alignat}
Hence if $u_3$ is any harmonic function which vanishes on $\partial M \backslash \Gamma$, then linearising the Dirichlet-Neumann maps yields the integral identity
\begin{alignat}{2}
0 & = \int_{\partial M} u_3 \partial_{\epsilon_1} \partial_{\epsilon_2} ( \Lambda_{X_1, V_1}^{\Gamma} - \Lambda_{X_2, V_2}^{\Gamma} ){f_{\epsilon}}_{|_{\epsilon = 0}} \hspace*{0.5mm} d\text{s}_{g} \nonumber \\
& = \int_{\partial M} u_3 \partial_{\nu} \partial_{\epsilon_1} \partial_{\epsilon_2} (u_{1,f_{\epsilon}} - u_{2,f_{\epsilon}})_{|_{\epsilon = 0}} d\text{s}_{g} + \int_{\partial M} 2i u_1 u_2 u_3 X_1(\nu) \hspace*{0.5mm} d\text{s}_{g}. \label{zeroth identity}
\end{alignat}
Integration by parts combined with (\ref{harmonic equation 1}) shows
\begin{alignat}{2}
& \int_{\partial M}  u_3  \partial_{\nu} \partial_{\epsilon_1} \partial_{\epsilon_2} (u_{1,f_{\epsilon}} - u_{2,f_{\epsilon}})_{|_{\epsilon=0}} \hspace*{0.5mm} d\text{s}_{g} \nonumber \\
& \hspace*{0.5mm} = \int_{M} \langle du_3, d \partial_{\epsilon_1} \partial_{\epsilon_2} (u_{1,f_{\epsilon}} - u_{2,f_{\epsilon}})_{|_{\epsilon = 0}} \rangle_{g} \hspace*{0.5mm} d\text{v}_{g} - \int_{M} u_3 \Delta_{g} \partial_{\epsilon_1} \partial_{\epsilon_2} (u_{1,f_{\epsilon}} - u_{2,f_{\epsilon}})_{|_{\epsilon = 0}} \hspace*{0.5mm} d\text{v}_{g} \nonumber \\
& \hspace*{0.5mm} = \int_{M} - \hspace*{0.5mm} 3i u_3 X_1 \wedge \star d(u_1 u_2) + u_1 u_2 u_3 \star (2i d^{\star} X_1 +V_2), \label{first identity}
\end{alignat}
where we have used that $ \partial_{\epsilon_1} \partial_{\epsilon_2} (u_{1, f_{\epsilon}} - u_{2,f_{\epsilon}})_{|_{\epsilon = 0}} $ have zero Dirichlet boundary condition and $u_3$ is harmonic.
\subsection{Boundary Determination} Now we claim that
\begin{alignat}{2} \label{integral identity to be used for boundary determination}
\int_{M} u_1 u_3 X_1 \wedge \star du_2 = \int_{M} u_1 u_2 X_1 \wedge \star du_3 
\end{alignat}
for all harmonic functions $u_1, u_2$ and $u_3$ which vanish on $\partial M \backslash \Gamma$. Indeed, combining (\ref{zeroth identity}) and (\ref{first identity}) we see that 
\begin{alignat}{2} \label{difference integral 1}
\int_{M} 3i u_3 X_1 \wedge \star d(u_1 u_2) = \int_{M} u_1 u_2 u_3 \star (2i d^{\star}X_1 + V_2) + \int_{\partial M} 2i u_1 u_2 u_3 X_{1}(\nu) \hspace*{0.5mm} d\text{s}_{g}.
\end{alignat}
By permuting the roles of $u_2$ and $u_3$, we also get 
\begin{alignat}{2} \label{difference integral 2}
\int_{M} 3i u_2 X_1 \wedge \star d( u_1 u_3 ) = \int_{M} u_1 u_2 u_3 \star (2i d^{\star} X_1 + V_1) + \int_{\partial M} 2i u_1 u_2 u_3 X_{1}(\nu) \hspace*{0.5mm} d\text{s}_{g}.
\end{alignat}
Taking the difference of (\ref{difference integral 1}) and (\ref{difference integral 2}), we have
\begin{alignat*}{2}
0 = \int_{M} 3i u_3 X_1 \wedge & \star d(u_1 u_2) - 3i u_2 X_1 \wedge \star d(u_1 u_3) \\
& = \int_{M} 3i u_1 u_3 X_1 \wedge \star du_2 - 3i u_1 u_2 X_1 \wedge \star du_3
\end{alignat*}
as claimed. Hence we can apply Proposition \ref{integral identity for boundary determination 2} to conclude that $uX_1(\tau) = uX_{1}(\nu) = 0$ along $\Gamma$ for any harmonic function $u$ vanishing on $\partial M \backslash \Gamma$. By choosing $u$ to be non-vanishing on the requires sets thus implies $X_1(\tau) = X_1(\nu) = 0$ along $\Gamma$. 
\subsection{Interior Determination}
By boundary determination, the difference of the linearised Dirichlet-Neumann maps reduces to 
\begin{alignat}{2} \label{reduced DN map}
\partial_{\epsilon_1} \partial_{\epsilon_2} (\Lambda_{X_1,V_1}^{\Gamma} - \Lambda_{X_2,V_2}^{\Gamma}){f_{\epsilon}}_{|_{\epsilon = 0}} = \partial_{\nu} \partial_{\epsilon_1} \partial_{\epsilon_2} ( u_{1,f_{\epsilon}} - u_{2,f_{\epsilon}} )_{|_{\Gamma}},
\end{alignat}
and so the integral identity coming from (\ref{zeroth identity}) becomes
\begin{alignat*}{2}
\int_{M} 3i u_3 X \wedge \star d(u_1 u_2 ) = \int_{M} u_1 u_2 u_3 \star (2i d^{\star} X_1 + V_1).
\end{alignat*}
Since $X(\nu) = 0$ along $\Gamma$, integrating by parts shows that
\begin{alignat*}{2}
\int_{M} 3i u_1 u_2 X_1 \wedge \star du_3 = \int_{M} u_1 u_2 u_3 \star ( i d^{\star} X_1 - V_2 ).
\end{alignat*}
Proposition \ref{identification} now implies $X_1 = 0$and $ id^{\star}X_1 = V_2 $. In particular $V_2 = 0$.
\subsection{The Induction Step} Now we assume that $X_{k'-1} = 0$ and $V_{k'} = 0$ for all $k' < k$ and $k>2$. We will show that $X_{k-1} = 0$ and $V_{k} = 0$ via the same strategy as above. Let $\epsilon = (\epsilon_1, ..., \epsilon_{k}) \in \mathbb{C}^{k}$ and $f_{\epsilon} = \epsilon_1 f_1 + \cdots + \epsilon_{k} f_{k}$ for $|\epsilon|$ sufficiently small and $f_1, ... , f_{k}$ in $\mathcal{C}^{2,\alpha}(\partial M) \cap \mathcal{E}'(\Gamma)$. Let $u_{j,f_{\epsilon}}$ be the unique small solution to $\Lambda_{X_j,V_j} u_{j,f_{\epsilon}} = 0$ with boundary condition $f_{\epsilon}$. A direct calculation as before shows that $\partial_{\epsilon_1} \cdots \partial_{\epsilon_{k}} {u_{j,f_{\epsilon}}}_{|_{\epsilon = 0}} = 0$ on $\partial M$ and
\begin{alignat*}{2}
\Delta_{g}  \partial_{\epsilon_1} \cdots \partial_{\epsilon_{k}} ( {u_{1,f_{\epsilon}} - u_{2,f_{\epsilon}})}_{|_{\epsilon = 0}} = (k+1) i & X_{k-1} \wedge \star d(u_1 \cdots u_{k})  - u_1 \cdots u_{k} (kid^{\star}X_{k-1} + V_k  )
\end{alignat*}
where $ u_1, ... , u_{k} $ are harmonic functions with Dirichlet conditions $u_{\ell} = f_{\ell}$. Likewise, we have
\begin{alignat*}{2}
\partial_{\epsilon_1} \cdots \partial_{\epsilon_k} {(  {u_{1,f_{\epsilon}} \langle X_{1}( u_{1,f_{\epsilon}}, p), \nu \rangle_{g} } - u_{2,f_{\epsilon}} \langle X_{2} (u_{2,f_{\epsilon}, p}), \nu \rangle_g )}_{|_{\epsilon=0}} = ki u_1 \cdots u_{k} X_{k-1}(\nu).
\end{alignat*}
Let $u_{k+1}$ be any harmonic function which vanish on $\partial M \backslash \Gamma$. Integration by parts now gives 
\begin{alignat*}{2}
0 & = \int_{\partial M} u_{k+1} ( \Lambda_{X_1, V_1}^{\Gamma} - \Lambda_{X_2,V_2}^{\Gamma} ){f_{\epsilon}}_{|_{\epsilon = 0}} \hspace*{0.5mm} d\text{s}_{g} = - \int_{M} (k+1)i u_{k+1} X_{k-1} \wedge \star d(u_1 \cdots u_{k})  \\
& \hspace*{2.3cm} + \int_{M} u_1 \cdots u_{k} u_{k+1} \star (kid^{\star}X_{k-1} + V_k)  +\int_{\partial M} ki u_1 \cdots u_{k} u_{k+1} X_{k-1}(\nu) \hspace*{0.5mm} d\text{s}_g.
\end{alignat*}
The same permutation argument leads to 
\begin{alignat*}{2}
\int_{M} u_1 \cdots u_{k-1} u_{k+1} X_{k-1} \wedge \star d{u_k} = \int_{M} u_1 \cdots u_{k-1} u_{k} X_{k-1} \wedge \star du_{k+1}.
\end{alignat*}
Hence by Proposition \ref{integral identity for boundary determination 2} we conclude that $X_{k-1}(\tau) = X_{k-1}(\nu) = 0$ along $\Gamma$. The boundary integral identity now reduces to
\begin{alignat*}{2}
\int_{M} (k+1)i u_{k+1} X_{k-1} \wedge \star d(u_1 \cdots u_{k}) = \int_{M} u_{1} \cdots u_{k} u_{k+1} \star (kid^{\star} X_{k-1} +V_{k}).
\end{alignat*}
Integration by parts lead to 
\begin{alignat*}{2}
\int_{M} (k+1)i u_{1} \cdots u_{k} X_{k-1} \wedge \star d u_{k+1} = \int_{M} u_1 \cdots u_{k} u_{k+1} \star (id^{\star}X_{k-1}-V_{k}).
\end{alignat*}
Hence Proposition \ref{identification} combined with the result of \cite{measurezero} implies $X_{k-1} = 0$ and $V_{k} = 0$ as before. By induction, $X_{k} = 0$ and $V_{k} = 0$ for all $k \geq 1$. It follows that $(X_1, V_1) = (X_{2},V_2)$ and we are done.
\end{proof}


\begin{thebibliography}{99}
 
\bibitem {5}
D. Dos Santos Ferreira, J. Sj\"ostrand, C.E. Kenig and G. Uhlmann, 
{Determining the magnetic field for the magnetic Schr\"odinger operator from partial Cauchy data},
{\it Communications in Mathematical Physics}, {\bf 271}(2), 2007, pages 467-488. 
 
\bibitem {CTANCP}
A. Feizmohammadi and L. Oksanen,
{An inverse problem for a semi-linear elliptic equation in Riemannian geometries}, 
{\it Journal of differential equations}, {\bf 269}(6), 2020, pages 4683-4719. 


\bibitem {14}
O. Imanuvilov, G. Uhlmann and M. Yamamoto,
{Partial Cauchy data for general second order operators in two dimensions}.
{\it Publications of the Research Institute for Mathematical Sciences}, {\bf 48}(4), 2010. 

\bibitem {guntherNMS}
K. Krupchyk and G. Uhlmann,
{Inverse problems for nonlinear magnetic Schr\"odinger equations on conformally transversally anisotropic manifolds},
{\it Preprint}, arXiv:2009.05089, 2020.

\bibitem {guntherNLC}
K. Krupchyk and G. Uhlmann,
{A remark on partial data inverse problems for semilinear elliptic equations,}
{\it Proceeding of the American Mathematical Society}, {\bf 148}, 2020, pages 681-685.
 
\bibitem {Lassasnonlinear} 
 M. Lassas, T. Liimatainen, Y.-H. Lin and M. Salo,
{Inverse problems for elliptic equations with power type nonlinearities},
to appear in {\it Journal de Mathematiques Pures et Appliquees}.   
 
\bibitem {LaiNMS}
R.-Y. Lai and T. Zhou, 
{Partial data inverse problems for nonlinear magntic Schr\"odinger equations}, 
{\it Preprint}, arXiv:2007.02475, 2020. 
 
\bibitem {measurezero}
B.S. Mityagin, 
{The zero set of a real analytic function},
{\it Matematicheskie Zametki}, {\bf 107}(3), 2020, pages 473-475. 

\bibitem {Ma}
Y. Ma and L. Tzou,
{Semilinear Calder\'on problem on Stein manifolds with K\"ahler metric},
to appear in {\it Bulletin of the Australian Mathematical Society}.
 
\bibitem {19}
G. Nakamura, Z.Q. Sun and G. Uhlmann,
{Global identifiability for an inverse problem for the Schrodinger equation n a magnetic field},
{\it Mathematische Annalen}, {\bf 303}(3), 1995, pages 377-388.
 
\bibitem {21} 
Z.Q. Sun, 
{An inverse boundary value problem for Schr\"odinger operators with vector potentials},
{\it Transactions of the American Mathematical Society}, {\bf 338}(2), 1993, pages 953-969.

\bibitem {22}
Z.Q. Sun,
{An inverse boundary value problem for the Schr\"odinger operator with vector potentials in two dimensions}, 
{\it Communication in Partial Differential Equations}, {\bf 18}(1-2), 1993, pages 83-124. 
 
\bibitem {calderonfirstresult}
J. Sylvester and G. Uhlmann, 
{A global uniqueness theorem for an inverse boundary value problem}, 
{\it Annals of Mathematics}, {\bf 125}(1), 1987, pages 153-169.	

\bibitem {reflectionpaper}
L. Tzou,
{The reflection principle and Calder\'on problems with partial data},
{\it Mathematische Annalen}, {\bf 369}(1-2), 2017, pages 913-956.

\bibitem {leocalderon}
L. Tzou and C. Guillarmou, 
{Calder\'on inverse problem with partial data on Riemann surfaces}, 
{\it Duke Mathematical Journal}, {\bf 158}(1), 2011, pages 83-120.

\bibitem {leoconnection}
L. Tzou and C. Guillarmou,
{Identification of a connection from Cauchy data on a Riemann surface with boundary},
{\it Geometric and Functional Analysis}, {\bf 21}(2), 2011, pages 393-418.



%\bibitem {beginNCP}
%Y. Kurylev, M. Lassa and G. Uhlmann,
%{\it Inverse problems for Lorentzian manifolds and non-linear hyperbolic equations},
%{\it Inventiones Mathematicae}, {\bf 212}(3), 2018, pages 781-857. 
 \end{thebibliography}
\end{document}